\documentclass[12pt]{amsart}

\usepackage[english]{babel}
\usepackage{amsmath,amssymb, amsthm,amsbsy} 
\usepackage{hyperref}  % [pdftex] if latex not used !!
\hypersetup{colorlinks=true, linkcolor=blue, anchorcolor=blue, 
citecolor=red, filecolor=blue, menucolor=blue,
urlcolor=blue}

\numberwithin{equation}{section}
\newtheorem{theorem}{Theorem}[section]
\newtheorem{corollary}[theorem]{Corollary}
\newtheorem{proposition}[theorem]{Proposition}
\newtheorem{lemma}[theorem]{Lemma}
\theoremstyle{remark}
\newtheorem{remark}{Remark}[section]

\theoremstyle{definition}
\newtheorem{definition}[theorem]{Definition}

\newcommand{\R}{{\mathbb R}}
 
\newcommand{\Z}{{\mathbb Z}}
\newcommand{\PP}{\mathcal{P}_l}
\newcommand{\C}{{\mathbb C}} 
\newcommand{\D}{\mathcal{D}}
\newcommand{\parmag}{\mathcal{D}_{A,l}}

\newcommand{\sol}{e^{it\mathcal{D}_A}f}
\newcommand{\ksol}{e^{it\mathcal{D}_A}F_l}

\newcommand{\four}{\mathcal{F}}

\newcommand{\heav}{\chi_{\mathbb{R}^+}}

\def\XXint#1#2#3{{\setbox0=\hbox{$#1{#2#3}{\int}$ }
\vcenter{\hbox{$#2#3$ }}\kern-.58\wd0}}

\begin{document}

\title
[Aharonov-Bohm field and Dirac equation]
{Dispersive estimates for the Dirac equation in an Aharonov-Bohm field }

\begin{abstract}
\end{abstract}

\date{\today}    %%% ''\date{}'' to omit date

%%%%%%%%%%%%%%%%%%%%   AUTHOR(S)                              (fold)

\author{F.~Cacciafesta}
\address{Federico Cacciafesta: Dipartimento di Matematica e Applicazioni, Universit$\grave{\text{a}}$ di Milano Bicocca, via R. Cozzi, 53, 20125 Milano, Italy.}
\email{federico.cacciafesta@unimib.it}

\author{L.~Fanelli}
\address{Luca Fanelli: SAPIENZA Universit$\grave{\text{a}}$ di Roma, Dipartimento di Matematica, P.le A. Moro 5, 00185 Roma, Italy.}
\email{fanelli@mat.uniroma1.it}


%\thanks{
%The two authors were supported by the Italian project FIRB 2012 {\it Dispersive Dynamics: Fourier Analysis and Calculus of Variations}.}

\begin{abstract}
We prove local smoothing and weighted Strichartz estimates for the Dirac equation with a Aharonov-Bohm potential. The proof, inspired by \cite{cacser}, relies on an explicit representation of the solution built in terms of spectral projections.
\end{abstract}

\date{\today}
\maketitle

\section{Introduction}

The Dirac Hamiltonian in the Aharonov-Bohm magnetic field (in the units with $h=c=1$) is
\begin{equation}\label{ham}
\mathcal{D}^m_A=\left(\begin{array}{cc}m & D^* \\D & -m\end{array}\right),\quad
D=(p_1+A^1)+i(p_2+A^2)
\end{equation}
where $p_j=i\partial_j$ and the magnetic potential $A$ reads as, in the radial gauge,
$$
A_r=0,\qquad A_\phi=\frac\alpha{2\pi r}\Phi_0.
$$ 
%In what follows, it will be useful to use the notation $[\alpha]$ and $\tilde{\alpha}$ to denote respectively the integer and the fractional part of $\alpha$, so that 
%$$\alpha=[\alpha]+\tilde{\alpha}.$$  
In terms of Pauli matrices, we can rewrite Hamiltonian \eqref{ham} as

%$\gamma^0=\beta=\sigma_3$, $\gamma^1=\sigma_2$ and $\gamma^2=-i\sigma_1$, where the $\sigma_j$ are the Pauli matrices
$$
\mathcal{D}^m_A(A)=\sigma_3 m+\sigma_1(p_1+A^1)+\sigma_2(p_2+A^2)
$$
where
\begin{equation}
\sigma_1=\left(\begin{array}{cc}0 & 1 \\1 & 0\end{array}\right),\quad
\sigma_2=\left(\begin{array}{cc}0 &-i \\i & 0\end{array}\right),\quad
\sigma_3=\left(\begin{array}{cc}1 & 0\\0 & -1\end{array}\right)
\end{equation}
and the magnetic potential $A(x)=(A_1(x),A_2(x))$ is given by
\begin{equation*}
A(x)=\alpha\left(-\frac{x_2}{|x|^2},\frac{x_1}{|x|^2}\right).
\end{equation*}
We recall that the Pauli matrices satisfy the following relations of anticommutations
\begin{equation*}
\sigma_j\sigma_k+\sigma_k\sigma_j=2\delta_{ik}\mathbb{I}_2,\quad j,k=1,2.
\end{equation*}
The Aharonov-Bohm effect was firstly predicted in \cite{ahbohm}: it occurs when electrons propagate in a domain with a zero magnetic field but with a nonzero vector potential $A^\mu$. The potential magnetic field is totally confined within a cylindrical tube of infinitesimal radius (see \cite{peston} and references therein for greater details).
In what follows we will restrict our attention to the massless case, i.e. $m=0$, and we will denote the corresponding Hamiltonian with $\mathcal{D}_A$. 

In the present paper we are interested in studying dispersive properties of the flow associated to the Hamiltonian $\mathcal{D}_A$: we thus mean to study the Cauchy problem 
\begin{equation}\label{diraceq}
\begin{cases}
\displaystyle
 i\partial_tu=\mathcal{D}_Au,\quad u(t,x):\mathbb{R}_t\times\mathbb{R}_x^2\rightarrow\mathbb{C}^{2}\\
u(0,x)=f(x).
\end{cases}
\end{equation}

We stress the fact that the magnetic potential $A$ is critical with respect to the scaling of the massless Dirac operator; as it is well known, the study of dispersive estimates for flows perturbed by scaling-critical potentials represents a particularly interesting and challenging problem, as perturbation arguments typically do not work in this setting. In this framework we mention \cite{burplan}-\cite{burplan2} in which smoothing and Strichartz estimates for the Schr\"odinger and wave equations with inverse square potential are proved, then \cite{fanfel} in which the $L^1\rightarrow L^\infty$ time decay is proved for a wide class of Schr\"odinger flows with critical electromagnetic potentials, later \cite{cacser} in which local smoothing estimates for the Dirac-Coulomb equation is discussed, and finally \cite{cacfan} which is devoted to the study of weak dispersion for fractional Aharonov-Bohm-Schr\"odinger groups. On the other hand, the study of dispersive estimates for the Dirac equation perturbed with small magnetic potential has been developed e.g. in \cite{bousdanfan} and \cite{cacmag}.

The first problem to be addressed in order to study the dynamics of equation \eqref{diraceq} is the self-adjointness of the Hamiltonian $\mathcal{D}_A$. As $\mathcal{D}_A$ commutes with the angular momentum operator, we can resort on the classical partial wave decomposition
$$L^2(\mathbb{R}^2)^{2} \cong L^2((0,\infty),dr)\bigotimes_{l\in\mathbb{Z}} h_l$$
 and reduce the problem to study the self-adjointness of "radial" Dirac operators $\parmag$ (see forthcoming Proposition \ref{thaller} and formula \eqref{partham} for details). By resorting on general theory (see Remark \ref{selfadj}) it is possible to show that the operators $\parmag$, which can be originally defined on the space $C^\infty_0(\mathbb{R}^+)$, are essentially self-adjoint as long as $l$ is {\em not} in the range $-1<l+\alpha<0$. This case needs to be discussed separately: it can be shown that the corresponding operator admits a one parameter family of self adjoint extensions, which can be distinguished by imposing different boundary conditions.

Before stating our main result, let us fix some useful notations. We shall denote as usual with $L^p_tL^q_x=L^p(\mathbb{R}_t; L^q(\mathbb{R}^n_x))$ the mixed space-time Strichartz spaces; with $L^2_{rdr}$ we will denote the $L^2$ space with respect to the measure $rdr$. The presence of the subscript $L^p_T$ will be a shortcut for $L^p_t([0,T])$, while we denote with $\|f\|_{L^2_{|x|\geq A}}:=\| f\:\chi_{|x|\geq A}\|_{L^2}$ where $\chi_A$ is the standard characteristic function of the set $A$.
%
% The Dirac operator with the Coulomb potential will be denoted by
%$$
%\mathcal{D}_\nu=\mathcal{D}+\frac{\nu}{|x|};
%$$
%we will need the fractional powers of this operator as well, that we will denote with $|\mathcal{D}_\nu|^\alpha$, that will be precisely defined the next section (see Proposition \ref{properties}). 
We denote with $\Omega^s$ the operator
\begin{equation*}
(\Omega^s\phi)(x)=|x|^s\phi(x)
\end{equation*}
and with a little abuse of notation we will use the same symbol to indicate the operators which are pointwise equal for all times,
\begin{equation*}
(\Omega^s\psi)(t,x)=|x|^s\psi(t,x).
\end{equation*}
We will also make use of the Gauss hypergeometric function, that we recall to be
$$
{_2F_1}(a,b;c;z)=\sum_{n=0}^\infty\frac{(a)_n(b)_n}{(c)_n}\frac{z^n}{n!}
$$
where here $(q)_n$ denotes the Pochhammer symbol. We refer to \cite{erd} as a general reference for details and properties of special functions.
\medskip

We are now ready to state the first result of this paper, that is a local smoothing estimate for solutions of \eqref{diraceq}; we prefer to postpone the precise definition of the  angular spaces $\mathcal{H}_{\geq \overline{l}}$ to forthcoming Proposition \ref{thaller}.
\begin{theorem}\label{thlocsmooth}
Let $\alpha\in\mathbb{R}$ and $u(t,x)$ be a solution of \eqref{diraceq}. Then for any 
\begin{equation}\label{range}
1/2<\gamma<1+|l+\alpha|.
\end{equation}
and any $f\in L^2((0,\infty)rdr)\otimes \mathcal{H}_{\geq \overline{l}}$ there exists a constant $c=c(\alpha,\gamma,l)$ such that the following estimate holds 
\begin{equation}\label{est}
\displaystyle
\left\|\Omega^{-\gamma}\mathcal{D}_A^{1/2-\gamma} u\right\|_{L^2_tL^2_x}\leq c \|f\|_{L^2}.
\end{equation}
In addition, in the endpoint case $\gamma=1/2$ the following estimate holds
\begin{equation}\label{locendec}
\sup_{R>0}R^{-1/2}\|\sol\|_{L^2_t L^2_{|x|\leq R}} \lesssim \|f\|_{L^2},
\end{equation}
\end{theorem}
From the point of view of the range in \eqref{range}, the worst frequency $l$ is the closest one to the circulation $\alpha\in\R$. Indeed, 
as an immediate consequence of \eqref{est}, we obviously obtain the following result for a generic non-localized $L^2$-function.
\begin{corollary}\label{cor:1}
Let $\alpha\in\mathbb{R}$, denote by $\mu_0:=\text{dist}(\alpha,\R)$ and let $u(t,x)$ be a solution of \eqref{diraceq}. Then for any 
\begin{equation}\label{range}
1/2<\gamma<1+\mu_0.
\end{equation}
and any $f\in L^2(\R^2)$ there exists a constant $c=c(\alpha,\gamma,l)$ such that the following estimate holds 
\begin{equation}\label{estestest}
\displaystyle
\left\|\Omega^{-\gamma}\mathcal{D}_A^{1/2-\gamma} u\right\|_{L^2_tL^2_x}\leq c \|f\|_{L^2}.
\end{equation}
\end{corollary}

\begin{remark}
Notice that estimate \eqref{estestest} for $\gamma=1$ fails in the free case (as it fails the corresponding one for the wave equation); on the other hand, as it is often the case, the presence of the Aharanov-Bohm magnetic field improves the range of the admissible exponents and allows us to include it, since $\mu_0\neq0$, as soon as $\alpha\notin\Z$. This is a kind of diamagnetic effect, which is in general not expected in the relativistic setting, but which is known to be possible (see \cite{AS}).  
\end{remark}

\begin{remark}
%We should stress that estimate \eqref{est} is the analog of Theorem 1.1 in \cite{cacser}, in which the same result is proved for the dynamics of the Dirac equation with a Coulomb potential.
It is interesting to compare the range of the admissible exponents \eqref{range} with its analogous for the local smoothing estimates for the Dirac-Coulomb model, as given in Theorem 1.1 in \cite{cacser}. In that case indeed, for generic initial data, the wideness of the range is a decreasing function of the (modulus of the) charge, meaning that the bigger the charge, the smaller the admissible range is, and it shrinks to the empty set as the charge tends to $1$, which we recall to be the threshold value for the charge in order to define a self-adjoint operator. Here instead, the range is in fact wider than the one of the free case. Also, estimate \eqref{locendec} could not (yet) be proved for Dirac-Coulomb: the main ingredient in our argument is indeed, as it will be clear, an estimate for the radial components of the generalized eigenstates $\chi_l(r)$ of the operator, namely
\begin{equation}\label{crucialprop}
\frac1R\int_0^R \chi_l(r)^2 rdr<C
\end{equation}
uniform in $R$ and $l$. In the contest of the Dirac equation with Aharonov Bohm fields, the functions $\chi_l(r)$ are nothing but standard Bessel functions $J_{\lambda}(r)$; property \eqref{crucialprop} with this choice has been originally proved in \cite{strichartz} for integers $k$ and then for generic $k$ in \cite{boh}.  It should also be noticed the fact that the stronger estimate
$$
\displaystyle
\sup_{r,\nu}|\sqrt{r}J_\lambda(r)|\leq C
$$
has been misproved in \cite{landbes}, where in particular the author shows that $\displaystyle\sup_{r>0} |\sqrt{r}J_\lambda(r)|$ strictly increases to infinity as $\lambda$ increases from $1/2$ to infinity. The case of Coulomb potential is more difficult, and this is due to the more complicated structure of generalized eigenstates, which involve confluent hypergeometric functions: the proof of estimate \eqref{crucialprop} in this setting will be object of forthcoming work.
\end{remark}

\begin{remark}
The problem, natural, of including a mass term in equation \eqref{diraceq} and thus in estimate \eqref{est} looks to be non trivial, as in our proof we strongly rely on scaling properties of the operator. A possible solution might be to exploit the Kato-Smoothing theory and then rely on the machinery developed in \cite{dancona}, and succesfully applied for Schr\"odinger flows in Aharonov-Bohm fields in \cite{cacfan}, to go from wave to Klein-Gordon smoothing estimates; nevertheless, the argument needs to be adapted to the delicate Dirac structure. This will be topic of forthcoming works.
\end{remark}

\begin{remark}
By combining this result with \cite{cacser}, in which an analogous estimate is proved for the massless Dirac-Coulomb system, it is possible to give an all-comprehensive estimate for an electromagnetic dynamic with Coulomb electric and Aharonov-Bohm magnetic potential. We mention that the explicit structure of the generalized eigenstates for this model can be found e.g. in \cite{khallee}.
\end{remark}

Next, we are able to prove the following local in time estimate, which in the contest of the wave equation is known as \emph{KSS estimate} (see \cite{kss}).
\begin{theorem}\label{Morawetz and KSS estimates}\label{ksstheo}
Let $\alpha\in\mathbb{R}$, $\mu\leq0$ and $u(t,x)$ be a solution of \eqref{diraceq}. For any $f\in L^2$ the following estimate is satisfied
%\begin{equation}\label{locendec}
%\sup_{R>0}R^{-1/2}\|\sol\|_{L^2_t L^2_{|x|\leq R}} \lesssim \|f\|_{L^2},
%\end{equation}
\begin{equation}\label{kss}
\|\langle x\rangle^\mu \sol\|_{L^2_TL^2_x}\lesssim A_\mu(T)\|f\|_{L^2}
\end{equation}
where
$$
A_\mu(T)=
\begin{cases}
T^{1/2-\mu},\quad {\rm if}\: -1/2<\mu\leq 0,\\
(\log((2+T))^{1/2},\quad  {\rm if}\: \mu=-1/2,\\
C \quad  {\rm if}\: \mu<-1/2.
\end{cases}
$$
where $C$ is an absolute constant.
Moreover, as a consequence of \eqref{locendec}, for $-1/2<\mu\leq 0$ we have
\begin{equation}\label{kss2}
\| |x|^\mu \sol \|_{L^2_TL^2_x}\lesssim T^{1/2+\mu}\| f\|_{L^2}.
\end{equation}
\end{theorem}

A natural application of Theorem \eqref{thlocsmooth} would be proving Strichartz estimates, which have been a relevant research topic in the last decades, due to their applications to nonlinear problems. The Aharonov-Bohm potential is known to be critical for their validity in this setting (see \cite{AFG} and also \cite{FG, GVV} for the non-relativistic counterpart). In order to apply a standard perturbation argument based on the Duhamel formula, one would need to use estimate \eqref{est} with $\gamma=1/2$, and this estimate seems to fail (surely we are not able to include it in our proof), as it fails in the free case. Anyway, as an application of Theorems \ref{thlocsmooth} and \ref{ksstheo}, it is still possible to prove some weighted Strichartz estimates. In the following, we use the polar coordinates $x=r\omega$, $r\geq0$, $\omega\in\mathbb S^1$, and given a measurable function $F=F(t,x):\R\times\R^2\to\C$ we denote by
$$
\|F\|_{L^q_tL^q_{rdr}L^2_\omega}
:=
\int_{-\infty}^{+\infty}\left(\int_{0}^{+\infty}\left(\int_{\mathbb S^1}|F(t,r,\omega)|^2\,d\sigma\right)^{q}\,rdr\right)\,dt,
$$
being $d\sigma$ the surface measure on the sphere. Then we can prove the following

\begin{theorem}\label{strichartz}
Let $\alpha\in\mathbb{R}$ and $u(t,x)$ be a solution of \eqref{diraceq}. For any $f\in L^2$ the following estimates are satisfied
\begin{equation}\label{weight}
\|r^{\frac12-\varepsilon-\frac2q}u\|_{L^q_tL^q_{rdr}L^2_\omega}\leq
C
\| \mathcal{D}_A^{\frac12+\varepsilon-\frac{1}{q}}\Lambda_\omega^{-\varepsilon+\frac\varepsilon{q}}f\|_{L^2},\qquad q\in[2,\infty]
\end{equation}
for some $C>0$, where $\Lambda_\omega=\sqrt{1-\Delta_\omega}$ and $\Delta_\omega$ is the Laplace-Beltrami operator on $S^1$.
%\begin{equation}\label{locintime}
%bbb
%\end{equation}
\end{theorem}

\begin{remark}
In this result we have used the fact that
$$
\| f\|_{\dot{H}^s}\leq C_1\|\mathcal{D}_A f\|_{L^2}
$$
for $s\in(0,1)$ (notice that fractional powers of $\mathcal{D}_A$ commute with the flow of \eqref{diraceq}). This will be proved in forthcoming Lemma \ref{confnorm}.
\end{remark}

\begin{remark}
By interpolating estimate \eqref{kss} with the endpoint trace lemma (see e.g. \cite{fangwang}), it is possible to obtain a family of local in time Strichartz estimates; anyway this would need a much careful insight in order to proper defined Besov spaces in this contest, and we prefer not to deal with this problem here.
\end{remark}

The proof of estimate \eqref{weight} only requires to interpolate between estimate \eqref{est} and the 2D Sobolev inequality 
$$
\sup_{r>0} r^{\frac12-\varepsilon}\|f(r\omega)\|_{L^2_\omega}\leq C\| |D|^{\frac12+\varepsilon}\Lambda_\omega^{-\varepsilon}f\|_{L^2_x}
\leq C\| \mathcal{D}_A^{\frac12+\varepsilon}\Lambda_\omega^{-\varepsilon}f\|_{L^2_x}.
$$
%Estimate \eqref{locintime} also is proved by interpolation, between estimates

%{\bf Acknowledgments.} We are grateful 

\section{The setup: spectral theory and integral transform.}

In this section we build the necessary setup needed to prove our main result.

\subsection{Spectral theory of the Dirac operator in Aharonov-Bohm field}
We devote this subsection to briefly review the spectral theory of the Dirac Hamiltonian in an Aharonov-Bohm field, as the explicit form of the generalized eigenstates will play a crucial role in what follows; for further details we refer to \cite{gerb}. 

First of all, we recall the following classical result, which can be found in \cite{thaller} (we will limit ourselves to the massless case as it is the one we need, but the theory can be extended to positive mass as well).
\begin{proposition}\label{thaller}
We can define a unitary isomorphism between Hilbert spaces 
$$L^2(\mathbb{R}^2)^{2} \cong L^2((0,\infty),dr)\bigotimes_{l\in\mathbb{Z}} h_l$$
by means of the following decomposition
\begin{equation}\label{iso2}
\Phi(x)=\sum_{l\in\mathbb{Z}}\frac{1}{2\sqrt{\pi}}\left(\begin{array}{cc} f_l(r)\\
 g_l(r) e^{i\phi}
\end{array}\right)e^{il\phi},
\end{equation}
which holds for any $\Phi\in L^2(\mathbb{R}^2,\mathbb{C}^2)$.
Moreover, the operator $\mathcal{D}^0_A$ defined in \eqref{ham} leaves invariant the partial wave subspaces $C^\infty_0((0,\infty))\otimes h_{l}$ and, with respect to the basis $\left\{e^{il\phi},e^{i(l+1)\phi}\right\}$ is represented by the matrix
\begin{equation}\label{partham}
\displaystyle
\parmag=\left(\begin{array}{cc}0 &-i\left(\partial_r+\frac{l+\alpha+1}r\right) \\-i\left(\partial_r-\frac{l+\alpha}r\right) & 0\end{array}\right).
\end{equation}
%\begin{equation}\label{raddir}
%\displaystyle
%d_l=\left(\begin{array}{cc}0& -\frac{d}{dr}+\frac{l+1/2}{r}
%\\  \frac{d}{dr}+\frac{l+1/2}{r} & 0\end{array}\right).
%\end{equation}
The operator  $\mathcal{D}^0_A$ on $C^\infty_0(\mathbb{R}^2)^2$ is unitary equivalent to the direct sum of $\parmag$ that is
\begin{equation*}
\mathcal{D}_A^0\cong \bigoplus_{l\in\mathbb{Z}}\parmag.
\end{equation*}
%and the restricted operators $d^n$ are selfadjoint if and only if the operator $\mathcal{D}_n-\frac{\nu}{r}$ are selfadjoint. 
For fixed $\overline{l}\in \mathbb{Z}$ we denote with
\begin{equation}\label{HH}
\mathcal{H}_{\geq \overline{l}}=\bigoplus_{l:|l|\geq|\overline{l}|}h_{l}.
\end{equation}
\end{proposition}

%By resorting on proposition above, after separation of variables in polar coordinates, it is seen that the Hamiltonian $\mathcal{D}_A$ defined in \eqref{ham} reduces to the direct sum 
%$$\mathcal{D}_A=\bigoplus_{l\in\mathbb{Z}} \D_{A,l}$$
% of the so called {\em channel operators} in $L^2((0,\infty)rdr)$, where
%\begin{equation}\label{partham}
%\displaystyle
%\parmag=\left(\begin{array}{cc}0 &-i\left(\partial_r+\frac{l+\alpha+1}r\right) \\-i\left(\partial_r-\frac{l+\alpha}r\right) & 0\end{array}\right).
%\end{equation}
Therefore, by defining for positive energies $E>0$
\begin{equation}\label{positen}
\Psi_{E,l}(t,r,\phi)\cong\left(\begin{array}{cc}f_{l,E}(r)\\g_{l,E}(r)e^{i\phi}\end{array}\right)e^{il\phi}e^{-iEt},
\end{equation}
the radial eigenvalue problem for a fixed value of $l\in\mathbb{Z}$ reads as 
\begin{equation}\label{eqspec}
\parmag\chi_{l,E}(r)=E\chi_{l,E}(r),
\end{equation} 
which gives the solution
\begin{equation}\label{contspec}
\chi_{l,E}(r)=\left(\begin{array}{cc}f_{l,E}(r)\\g_{l,E}(r)\end{array}\right)=
\displaystyle
\sqrt\frac\pi2\left(\begin{array}{cc}(\epsilon_l)^lJ_{|l+\alpha|}(Er) \\ i(\epsilon_l)^{l+1}J_{|l+1+\alpha|}(Er)\end{array}\right)
\end{equation}
with
$$
\epsilon_l=\begin{cases}
1\qquad {\rm if}\:l+\alpha\geq0\\
-1\quad\:{\rm if}\:l+\alpha<0.
\end{cases}
$$
%\textcolor{red}{On the other hand, it is readily seen that $\parmag^*=-\parmag$; therefore, the generalized eigenstates for negative energies can be directly obtained as
%$$
%\psi
%$$}

Direct calculations show that the generalized eigenfunctions for negative values of the energy can be written as
\begin{equation}\label{negen}
\chi_{l,-E}(r)=\overline{\chi_{l,E}}=\sqrt\frac\pi2\left(\begin{array}{cc}(\epsilon_l)^lJ_{|l+\alpha|}(|E|r) \\ -i(\epsilon_l)^{l+1}J_{|l+1+\alpha|}(|E|r)\end{array}\right),
\end{equation}
so that in particular
$$
f_{l,-E}(r)=f_{l,E}(r),\qquad g_{l,-E}(r)=-g_{l,E}(r).
$$
Notice that the wave functions above can be normalized by the condition
\begin{equation}\label{normaliz}
\int \chi^*_{l,E}(x)\chi_{l',E'}(x)dx^2=2\pi \delta_{l,l'}\frac{\delta(E-E')}E.
\end{equation}
\begin{remark}\label{selfadj}
Let us briefly comment on the formulas above. System \eqref{eqspec} admits indeed, for a fixed value of $l$, another family of solutions that can be obtained by replacing the couple $(J_{|l+\alpha|}, J_{|l+\alpha+1|})$ by $(J_{-|l+\alpha|},J_{-|l+\alpha+1|})$ in \eqref{contspec}. On the other hand, this second couple can not be considered, as long as $|l+\alpha|>1$, as generalized eigenstates are required to be square-integrable in the origin; we recall indeed the well known asymptotic behavior of Bessel functions in the origin given by
$$\lim_{x\rightarrow 0}J_\nu(x)\cong\frac{x^\nu}{2^\nu\Gamma(1+\nu)}.$$
General theory therefore ensures that the operators $\parmag$ are essentially selfadjoint as long as $|l+\alpha|>1$.
When $-1<l+\alpha<0$, the situation becomes more delicate: both the couples $(J_{\pm|l+\alpha|},J_{\pm|l+\alpha+1|})$ satisfy indeed the condition of square integrability in the origin. This is enough to guarantee that the corresponding operator $\parmag$ is \emph{not} essentially selfadjoint (recall that a corollary of the Theorem proven  in \cite{weidmann} states that,  for the partial Dirac Hamiltonian to be
essentially self-adjoint, it is necessary and sufficient that a non-square-integrable at $r\rightarrow 0$ solution exists). On the other hand, by taking advantage of the classical Von Neumann deficiency indices theory, it is possible to show that the operator $\parmag$ in the "critical case" $-1<l+\alpha<0$ admits a one-parameter family of self adjoint extensions; to fix one, suitable futher boundary conditions on the eigenfunctions need to be imposed. We refer to \cite{sitenko}, \cite{gerb} and \cite{falomir} for a detailed analysis of the topic.
\end{remark}

\subsection{The integral transform}
We now introduce the crucial integral transform that will be used in the proof of the main result, that essentially consists in a projection on the continuous spectrum. Throughout this subsection, we are fixing a value of $l\in\mathbb{Z}$ and working only on radial functions.

\begin{definition}\label{hanktras}
Let $\varphi(r)=(\varphi_1(r),\varphi_2(r))\in L^2((0,\infty),r dr)^2$.  We define, for $l\in\mathbb{Z}$, the following integral transform
\begin{equation}\label{H}
\mathcal{P}_l\varphi(E)=
\left(\begin{array}{cc}\mathcal{P}^+_l\varphi(E)\\
\mathcal{P}^-_l\varphi(E)
\end{array}\right)=
\left(\begin{array}{cc}\int_0^{+\infty}\chi_{l,E}(r)\varphi(r)rdr\\
\int_0^{+\infty}-\chi_{l,-E}(r)\varphi(r)rdr
\end{array}\right)
\end{equation}
\begin{equation*}
=\int_0^{+\infty}H_{l}(\varepsilon r)\cdot\varphi(r)rdr
\end{equation*}
where we have introduced the matrix
\begin{equation}\label{matra}
H_{l}=\left(\begin{array}{cc}f_{l,E}(r)\;g_{l,E}(r)\\
-f_{l,-E}(r)\;-g_{l,-E}(r).
\end{array}\right)
\end{equation}
\end{definition}

We collect in the following proposition some crucial properties of the operator $\mathcal{P}_l$. \begin{proposition}\label{properties}
For any $\varphi\in L^2((0,\infty),dr)^2$ the following properties hold:
\begin{enumerate}
\item
$\mathcal{P}_l$ is an $L^2$-isometry.
\item
$
\mathcal{P}_l \parmag=\Omega\mathcal{P}_l.
$
\item
The inverse transform of $\mathcal{P}_l$ is given by
\begin{equation}\label{H-1}
\mathcal{P}_l^{-1}\varphi(r)=\int_0^{+\infty}H_{l}^{*}(E r)\cdot\varphi(E)E dE
%\mathcal{K}_{lm}^+\Phi(r):=\int_0^{+\infty}\overline{\psi(\epsilon r)}\cdot \phi(\epsilon)\epsilon^2d\epsilon,\quad
%\mathcal{K}_{lm}^-\Phi(r):=\int_0^{+\infty}\overline{\tilde{\psi}(\epsilon r)}\cdot \phi(\epsilon)\epsilon^2d\epsilon
\end{equation}
where the matrix $H_{l}^{*}$ is the transpose conjugate of $H_l$.
%The inverse transform of $\mathcal{P}_l$ is given by
%\begin{equation}\label{H-1}
%\mathcal{P}_l^{-1}\varphi(r)=\int_0^{+\infty}H_{l}^{*}(E r)\cdot\varphi(E)EdE
%%\mathcal{K}_{lm}^+\Phi(r):=\int_0^{+\infty}\overline{\psi(\epsilon r)}\cdot \phi(\epsilon)\epsilon^2d\epsilon,\quad
%%\mathcal{K}_{lm}^-\Phi(r):=\int_0^{+\infty}\overline{\tilde{\psi}(\epsilon r)}\cdot \phi(\epsilon)\epsilon^2d\epsilon
%\end{equation}
%where $H_{l}^{*}=\left(\begin{array}{cc}f_{l,E}(r)\;\;f_{l,-E}(r)\\ g_{l,E}(r)\;\;
%g_{l,-E}(r)
%\end{array}\right)$.
\item
For every $\gamma\in\mathbb{R}$ we can formally define the fractional operators
\begin{equation}\label{fractiondef}
\parmag^\gamma\varphi_l(r)=\mathcal{P}_l\Omega^\gamma\mathcal{P}_l^{-1}\varphi_l(r)=\int_0^{+\infty}S_l^\gamma(r,s)\cdot \varphi_l(s)sds.
\end{equation}
where the integral kernel $S_l(r,s)$ is the $2\times2$ matrix given by
\begin{equation}\label{mattrix}
S_l^\gamma(r,s)=\int_0^{+\infty}H_{l}(E r)\cdot H^*_{l}(E s)E^{1+\gamma}dE
\end{equation}
\end{enumerate}
\end{proposition}
%
%\begin{remark}
%Property $(1)$ allows, by standard arguments, to extend the definition of the operator $\mathcal{P}_l$ to functions in $L^2$.
%\end{remark}

\begin{remark}\label{fractpow}
When summing on $l$, property \eqref{fractiondef} defines in a standard way fractional powers of the Hamiltonian $\mathcal{D}_A$, which are used in the statement of Theorem \eqref{thlocsmooth}.
\end{remark}

\begin{remark}\label{hankel}
Let us notice that $\mathcal{P}_l^{+}$ and $\mathcal{P}_l^{-}$ are (the radial part of) a sum of Hankel transforms: indeed, due to \eqref{contspec}, we have that (some factors $\pi$ will be neglected)
$$
\mathcal{P}_l^{+}\phi(E)=\int_0^\infty \big(J_{|l+\alpha|}(Er)\phi_1(r)+J_{|l+1+\alpha|}(Er)\phi_2(r)\big)r dr
$$
and, due to \eqref{negen}, a similar one for $\mathcal{P}_l^{-}$. In Proposition above we are thus just transferring to our framework several important properties that are well known for Hankel (see e.g. \cite{burplan}).

\end{remark}

\begin{proof}

Property $(1)$ is a standard feature of Hankel transform. 

Property $(2)$ comes from the definition of $\mathcal{P}_l$, once noticed that
$$
\mathcal{P}^+_l(\parmag\varphi)=\langle \chi_{l, E},\parmag\varphi\rangle_{L^2_{rdr}}=\langle \parmag \chi_{l, E},\varphi\rangle_{L^2_{rdr}}= E\langle  \chi_{l, E},\varphi\rangle_{L^2_{rdr}}
$$
where we have used the fact that $\parmag$ is selfadjoint with respect to the scalar product above (an analogous calculation can be developed for $\mathcal{P}^-$). This shows that
$$
\mathcal{P}_l^{\pm}(\parmag\varphi)(E)=E\mathcal{P}_l^{\pm}\varphi(E),
%\mathcal{P}_l^-(\parmag\varphi)(E)=-E\mathcal{P}_l^-\varphi(E)
$$
and thus
$$
\mathcal{P}_l(\parmag\varphi)(E)=\left(\begin{array}{cc}\mathcal{P}^+_l\parmag\varphi(E)\\
\mathcal{P}^-_l\parmag\varphi(E)
\end{array}\right)=
\left(\begin{array}{cc}E\mathcal{P}^+_l\varphi(E)\\
E\mathcal{P}^-_l\varphi(E)
\end{array}\right)= \Omega \varphi(E)
$$
which proves Property $(2)$.

Property $(3)$ is a direct calculation.
% consequence of $(1)$ and general theory, as $\mathcal{P}_l$ is an $L^2$ isometry and a unitary operator, and \eqref{oss}: we have indeed that
%$$
%H_{l}^{*}=\left(\begin{array}{cc}f_{l,E}(r)\;\;f_{l,-E}(r)\\ g_{l,E}(r)\;\;
%g_{l,-E}(r)
%\end{array}\right)=
%\left(\begin{array}{cc}f_{l,E}(r)\;g_{l,E}(r)\\
%f_{l,-E}(r)\;g_{l,-E}(r).
%\end{array}\right)
%$$
%and thus $\mathcal{P}_l^2=$Id.

To prove property $(4)$ we write
\begin{eqnarray*}
\parmag^\gamma\varphi_l(r)&=&\mathcal{P}_l\Omega^\gamma\mathcal{P}_l\varphi_l(r)
\\
&=&
\int_0^{+\infty}H_{l}(E r)E^{1+\gamma}\left(\int_0^{+\infty}H^*_{l}(E s)\varphi_l(s)sds\right)dE.
\end{eqnarray*}
Exchanging the order of the integrals yields \eqref{fractiondef}-\eqref{mattrix} and thus $(4)$.
%
%\forall\Phi\quad\langle\mathcal{H}_{lm}^\pm\mathcal{K}_{lm}^\pm\Phi,\psi_{\tilde\epsilon}(r)\rangle=\langle \Phi,\psi_{\tilde\epsilon}(r)\rangle
%$$
%which is true because of condition \eqref{normeigen}.
\end{proof}

\begin{remark}
Notice that, due to \eqref{negen}, 
\begin{equation*}
H_{l}^*=\left(\begin{array}{cc}f_{l,E}(r)\;-f_{l,E}(r)\\
g_{l,E}^*(r)\;g_{l,E}^*(r)
\end{array}\right).
\end{equation*}
\end{remark}
By calculating the integrals, we are able to write down explicitly the single components of the $2\times 2$ integral kernel $S^\gamma_l(r,s)$.

\begin{proposition}\label{compon}
Let $l\in\mathbb{Z}$, $\gamma>0$, $0<r<s$. Then,
\begin{equation}\label{expform}
S_{l}^{\gamma}=\left(\begin{array}{cc}F^\gamma_l(r,s)\;\;G^\gamma_l(r,s)\\ G^\gamma_l(r,s)\;\;
F^\gamma_l(r,s)
\end{array}\right)
\end{equation}
where
\begin{equation}\label{F_1}
\displaystyle
F_l^\gamma(r,s)=A+B,\qquad  G_l^\gamma(r,s)=-A+B
\end{equation}
with
\begin{equation*}
A=\frac{2^{\gamma}\pi\Gamma\left(|l+\alpha|+\frac\gamma2+1\right)}{\Gamma(-\frac{\gamma}2)\Gamma(|l+\alpha|+1)}\frac{r^{|l+\alpha|}}{s^{|l+\alpha|+\gamma+2}}\:
{_2F_1}\left(|l+\alpha|+\frac\gamma2+1, \frac\gamma2+1; |l+\alpha|+1;\frac{r^2}{s^2}\right)
\end{equation*}
and
\begin{equation*}
B=\frac{2^{\gamma}\pi\Gamma\left(|l+\alpha|+\frac{\gamma}2+2\right)}{\Gamma(-\frac{\gamma}2)\Gamma(|l+\alpha|+2)}\frac{r^{|l+1+\alpha|}}{s^{|l+\alpha|+\gamma+3}}
{_2F_1}\left(|l+\alpha|+\frac\gamma2+2, \frac\gamma2+1; |l+\alpha|+2;\frac{r^2}{s^2}\right).
\end{equation*}
%\begin{equation}\label{G}
%\displaystyle
%G_l^\gamma(r,s)=
%\end{equation}
%\begin{equation*}
%\frac{2^{\gamma}\pi\Gamma\left(|l+\alpha|+\frac{\gamma+3}2\right)}{\Gamma(\frac{1-\gamma}2)\Gamma(|l+\alpha|+1)}\frac{r^{|l+\alpha|}}{s^{|l+\alpha|+\gamma+2}}
%{_2F_1}\left(|l+\alpha|+\frac\gamma2+1, \frac{\gamma+1}2; |l+\alpha|+1;\frac{r^2}{s^2}\right)
%\end{equation*}
%\begin{equation*}
%+\frac{2^{\gamma}\pi\Gamma\left(|l+\alpha|+\frac{\gamma+3}2\right)}{\Gamma(\frac{-1-\gamma}2)\Gamma(|l+\alpha|+2)}\frac{r^{|l+1+\alpha|}}{s^{|l+\alpha|+\gamma+3}}
%{_2F_1}\left(|l+\alpha|+\frac\gamma2+1, \frac{\gamma+3}2; |l+\alpha|+2;\frac{r^2}{s^2}\right).
%\end{equation*}

%\textcolor{red}{SAREBBE BELLO SE $G_l^\gamma$ FOSSE ZERO, MA MI PARE DI NO...}
\end{proposition}

\begin{remark}
The representation in the region $0<s<r$ can be obtained by exchanging the roles of $r$ and $s$ in formulas above.
\end{remark}

\begin{remark}
This result should be compared with formula $(10)$ in \cite{burplan} (see also \cite{planc}).
\end{remark}

\begin{proof}
We rely on explicit formulas to calculate the integrals in \eqref{expform}, which have already been object of detailed analysis in the literature. We recall indeed the general results (see e.g. \cite{mag} pag. 49)
\begin{equation}\label{intbess}
\displaystyle
\int_0^\infty J_\nu(rt) J_\mu(st) t^{-\lambda} dt=
\frac{r^\nu \Gamma\left(\frac{\nu+\mu-\lambda+1}2\right)}{2^\lambda s^{\nu-\lambda+1}\Gamma\left(\frac{-\nu+\mu+\lambda+1}2\right)\Gamma(\nu+1)}
\end{equation}
\begin{equation*}
\times {_2F_1}\left(\frac{\nu+\mu-\lambda+1}2, \frac{\nu-\mu-\lambda+1}2; \nu+1;\frac{r^2}{s^2} \right),
\end{equation*}
provided ${\rm Re}(\nu+\mu-\lambda+1)>0$, ${\rm Re} (\lambda)>-1$ and $0<r<s$. We rely on this formula to evaluate our integrals. We have indeed that
\begin{eqnarray*}
F_l^\gamma(r,s)&=&\int_0^\infty (f_{l,E}(r)f_{l,E}(s)+g_{l,E}(r)g_{l,E}^*(s)) E^{1+\gamma}dE
\\
&=&
\int_0^\infty
\big(J_{|l+\alpha|}(Er)J_{|l+\alpha|}(Es)+
J_{|l+1+\alpha|}(Er)J_{|l+1+\alpha|}(Es)\big) E^{1+\gamma}dE.
\end{eqnarray*}
Applying \eqref{intbess} with the choice $\nu=\mu=|l+\alpha|$ (resp. $\nu=\mu=|l+1+\alpha|$) and $\lambda=-1-\gamma$ (notice that our assumptions on the parameters allow us to rely on such a formula), gives \eqref{F_1}. Analogously one obtains that
\begin{eqnarray*}
G_l^\gamma(r,s)&=&\int_0^\infty (-f_{l,E}(r)f_{l,E}(s)+g_{l,E}(r)g_{l,E}^*(s)) E^{1+\gamma}dE
\\
&=&
\int_0^\infty
\big(-J_{|l+\alpha|}(Er)J_{|l+\alpha|}(Es)+
J_{|l+1+\alpha|}(Er)J_{|l+1+\alpha|}(Es)\big) E^{1+\gamma}dE.
\end{eqnarray*}
Applying \eqref{intbess} concludes the proof.
\end{proof}

\subsection{The norm induced by $\mathcal{D}_A$}

We conclude this section with the following Lemma, which is a key ingredient for proving Theorem \ref{strichartz}.

\begin{lemma}\label{confnorm}
Let $\alpha\in \mathbb{R}$. For any $s\in[0,1]$ 
\begin{equation}\label{normconf1}
\| f\|_{\dot{H}^s}\leq C_1\|\mathcal{D}_A f\|_{L^2}.
\end{equation}
\end{lemma}

\begin{proof}
We prove the equivalence $\|\nabla_Af\|_{\dot{H}^s}\cong\|\mathcal{D}_A f\|_{L^2}$, where we are denoting with 
$$
\nabla_A=(\partial_1^A,\partial_2^A)=(\partial_1+iA_1(x),\partial_2+iA_2(x))
$$ 
the magnetic gradient; then estimate \eqref{normconf1} will be a consequence of diamagnetic inequality (see e.g. \cite{fanfel2, fanfel3, FGK}). Moreover, it is enough to prove the case $s=1$, as the full range of exponents can then be obtained by interpolation (the case $s=0$ is obvious). But this is an immediate consequence of the relation of anticommutations of Pauli matrices as indeed
\begin{eqnarray*}
\|\mathcal{D}_A f\|_{L^2}^2&=&\int |(\sigma_1\partial^1_A+\sigma_2\partial^2_A)f|^2
\\
&=&
\int \big|\big[(\sigma_1\partial^1_A)^2+(\sigma_2\partial^2_A)^2+(\sigma_1\partial_1^A\sigma_2\partial_2^A+
\sigma_2\partial_2^A\sigma_1\partial_1^A)\big]f\big|
\\
&=&
\int|\nabla_A f|^2
\end{eqnarray*}
which concludes the proof.
%\textcolor{red}{TO BE FINISHED}
%After some easy calculations we obtain, denoting with $\langle\cdot,\cdot\rangle$ the standard $L^2$ product,
%\begin{eqnarray*}
%\||\mathcal{D}_A| f\|_{L^2}^2&=&
%\langle \mathcal{D}_A f,\mathcal{D}_A \rangle
%\\
%&=&
%\langle f,-\Delta_A f\rangle+\sigma_1\sigma_2(\partial_1A^2-\partial_2A^1)
%\end{eqnarray*}
%where $\Delta_A=(\nabla-iA(x))^2$. On the other hand, it is seen that for the Aharonov-Bohm field the term $\partial_1A^2-\partial_2A^1=0$, and this concludes the proof.
%\textcolor{red}{ATTENZIONE COMPARE LA DELTA IN ZERO}
\end{proof}

\section{Proof of the main results}

We devote this section to proving our main Theorems.

\subsection{Proof of Theorem \ref{thlocsmooth}}
Our proof is a combination of the arguments used in \cite{cacser} with the ones in \cite{cacfan}, and follows the strategy originally developed in \cite{burplan}, the idea being use decomposition \eqref{iso2} to reduce equation \eqref{diraceq} to a much simpler problem, use Propositions \eqref{properties} and \eqref{compon} to prove the local smoothing estimate for a fixed value of $l\in\mathbb{Z}$ and then sum back. We thus set an initial condition $f\in L^2$ with angular part in $h_l$ and denote with $L_l f$ the solution to the Cauchy problem
\begin{equation}\label{diraccoulmod}
\begin{cases}
\displaystyle
 i\partial_tu=\parmag u,\\
u(0,x)=f(x)
\end{cases}
\end{equation}
where $\parmag$ is given by \eqref{partham}. Then, by applying operator $\mathcal{P}_l$ and using its properties, the LHS of estimate \eqref{est} is equivalent to (notice that the application of the matrix $\sigma_3$ does not alter the $L^2$ norm)
$$
\| \PP \Omega^{-\gamma} \parmag^{1/2-\gamma} L_l f\|_{L^2_tL^2_x}=\| \parmag^{-\gamma} \Omega^{1/2-\gamma}\PP L_l f\|_{L^2_tL^2_x}
$$
The function $\mathcal{P}_lL_lf$ solves now
\begin{equation}\label{diracproj}
\begin{cases}
\displaystyle
 i\partial_t\mathcal{P}_lL_lf=\Omega \mathcal{P}_lL_lf,\\
\mathcal{P}_lL_lf(0,\xi)=\mathcal{P}_lf(\xi),
\end{cases}
\end{equation}
so that the solution to this problem is explicitly given by
\begin{equation*}
\mathcal{P}_kL_kf(t,\xi)=e^{it\xi}\mathcal{P}_kf(\xi).
\end{equation*}
We now Fourier transform in time (which does not alter the $L^2$ norm) to have
\begin{equation*}
(\mathcal{F}_t\mathcal{P}_lL_lf)(\tau,\xi)=(\mathcal{P}_lf)(\rho)\delta(\tau+\xi).
\end{equation*}
Therefore, we can write
\begin{eqnarray*}
(\parmag^{-\gamma}\Omega^{1/2-\gamma}\mathcal{F}_t\mathcal{P}_lL_lf)(\tau,\xi)&=&
\int_0^{+\infty}S_l^{-\gamma}(\xi,s)\delta(\tau+s)\mathcal{P}_lf(s)s^{\frac{3-2\gamma}2}ds
\\
&=&
-S_l^{-\gamma}(\xi,\tau)\mathcal{P}_lf(\tau)\tau^{\frac{3-2\gamma}2}.
\end{eqnarray*}
We now take the $L^2$ norm in time and space of quantity above  (notice that, as the angular part in decomposition \eqref{iso2} is $L^2$-unitary, we only need to consider the radial integrals)
\begin{equation}\label{1}
\int_0^{+\infty}\int_0^{+\infty}((\mathcal{P}_lf)^*(\tau)S_l^{-\gamma}(\rho,\tau)^T)\cdot(S_l^{-\gamma}(\rho,\tau)(\mathcal{P}_lf)(\tau))\tau^{3-2\gamma}\rho^2 d\rho.
\end{equation}
Since $S_l^p(\rho,\tau)^T=S_l^p(\tau,\rho)$, the integral in $d\rho$ yields $S_l^{-2\gamma}(\tau,\tau)$ and we are therefore left with
\begin{equation*}
\int_0^{+\infty}(\mathcal{P}_lf)^*(\tau)S_l^{-2\gamma}(\tau,\tau)(\mathcal{P}_lf)(\tau)\tau^{3-2\gamma} \:d\tau
\end{equation*}
\begin{equation}\label{2}
\leq
\int_0^{+\infty} {\rm Tr}(S_l^{-2\gamma}(\tau,\tau))|(\PP f)(\tau)|^2\tau^{3-2\gamma} \:d\tau
\end{equation}
where in the last step we have used the fact that the matrix $S^{-\gamma}_l$ is positive definite in the diagonal values, as it is the integral kernel of a positive definite operator.
To conclude the proof we thus need to estimate the integral above: in view of Proposition \eqref{compon} we recall the explicit formula of Gauss hypergeometric functions with argument $1$, to be
\begin{equation}\label{specva}
{_2F_1}(a,b;c;1)=\frac{\Gamma(c)\Gamma(c-a-b)}{\Gamma(c-a)\Gamma(c-b)}
\end{equation}
provided ${\rm Re}(c-a-b)>0$ (notice that this restriction forces the bound $\gamma>1/2$);
we also stress that when $\tau:=r=s$ we have the equivalence of the ratios in \eqref{F_1}
$$
\frac{r^{|l+\alpha|}}{s^{|l+\alpha|-2\gamma+2}}=\frac{r^{|l+1+\alpha|}}{s^{|l+\alpha|-2\gamma+3}}=
\tau^{2\gamma-2}
$$
which is the right weight which allows to recover the $L^2$-norm in \eqref{2}. Therefore, we eventually have 
\begin{eqnarray*}
\eqref{2}\leq C_{\gamma,\alpha,l}\int_0^{+\infty}|(\PP f)(\tau)|^2\tau \:d\tau\leq C_{\gamma,\alpha,l}\int_0^{+\infty}|f(\tau)|^2\tau \:d\tau=C_{\gamma,\alpha,l}\| f\|_{L^2}
\end{eqnarray*}
where the constant
\begin{equation}\label{constant}
C_{\gamma,\alpha,l}=\frac{\pi\Gamma(2\gamma-1)}{2^{2\gamma}\Gamma(\gamma)^2}\left[\frac{\Gamma(|l+\alpha|-\gamma+1)}{\Gamma(|l+\alpha|+\gamma)}+
\frac{\Gamma(|l+\alpha|-\gamma+2)}{\Gamma(|l+\alpha|+\gamma+1)}\right]
%\\
%&=&
%\nonumber
%\frac{\pi\Gamma(2\gamma-1)}{2^{2\gamma}\Gamma(\gamma)^2}\left[
%\frac{\Gamma(|l+\alpha|-\gamma+1)}{\Gamma(|l+\alpha|+\gamma)}+\frac{\Gamma(|l+\alpha|-\gamma+1)(|l+\alpha|-\gamma+1)}{\Gamma(|l+\alpha|+\gamma)(|l+\alpha|+\gamma)}
%\right]
%\\
%&=&
%\nonumber
%\frac{\pi\Gamma(2\gamma-1)\Gamma(|l+\alpha|-\gamma+1)}{2^{2\gamma}\Gamma(\gamma)^2}\left[
%\frac{(|l+\alpha|+\gamma)+(|l+\alpha|-\gamma+1)}{\Gamma(|l+\alpha|+\gamma+1)}
%\right]
%\\
%&=&
%\frac{\pi\Gamma(2\gamma-1)\Gamma(|l+\alpha|-\gamma+1)}{2^{2\gamma}\Gamma(\gamma)^2\Gamma(|l+\alpha|+\gamma+1)}\left(
%2|l+\alpha|+1
%\right)
\end{equation}
Notice that our assumption on the range of $\gamma$ is now necessary to guarantee the constant $C_{\gamma,\alpha,l}$ to be finite, as indeed we are forced to assume $\gamma>1/2$ and $\gamma<|l+\alpha|+1$. This proves the inequality for a fixed level $l$. Notice also that within our range $C_{\gamma,\alpha,l}$ is in fact a decreasing function of $l$; this allows us to rely on decomposition \eqref{iso2} and use triangle inequality to conclude the proof.

We now turn to the proof of estimate \eqref{locendec}.

In analogy with what has been done above, we resort on decomposition \eqref{iso2} and prove it for a fixed value of $l$, showing that in fact the estimate holds with a constant independent on $l$: this will allow to sum back and obtain \eqref{locendec}. We denote for a fixed value of $l\in\mathbb{Z}$
$$
F_l(r,\phi)=\left(\begin{array}{cc} f_l(r)e^{il\phi}\\
i g_l(r) e^{i(l+1)\phi}
\end{array}\right),\quad F_l(r)=\left(\begin{array}{cc} f_l(r)\\
i g_l(r) 
\end{array}\right),\quad
Y_l(\phi)=
\left(\begin{array}{cc}e^{il\phi}\\
e^{i(l+1)\phi}
\end{array}\right)
$$ 

Writing indeed
\begin{eqnarray*}
\|\sol\|_{L^2_t L^2_{|x|\leq R}}^2&=&
\|\sum_{l\in\mathbb{Z}}\ksol(r,\phi)\||_{L^2_t L^2_{|x|\leq R}}^2
\\
&\leq&
\sum_{l\in\mathbb{Z}}\| \ksol (r,\phi)\|_{L^2_t L^2_{|x|\leq R}}^2
\\
&=&
\sum_{l\in\mathbb{Z}}
\|e^{it\mathcal{D}_A}F_l(r)\|_{L^2_{rdr}(0,R)}^2\|Y_l(\phi)\|_{L^2_\phi}^2
\end{eqnarray*}
so that, by the unitarity of the angular term, it will be enough to prove
$$
\sup_{R>0}\|e^{it\mathcal{D}_A}F_l(r)\|_{L^2_{rdr}(0,R)}^2\leq C_l\|F_l(r) \|_{L^2_{rdr}}
$$
and show that the constant $C_l$ is bounded with respect to $l$.
We start by writing
\begin{eqnarray*}
\ksol&=&\PP^{}\left[e^{it|\xi|}\PP f\right]
\\
&=&
\int_0^{+\infty} e^{it|\xi|}H_l(r\xi)\cdot \PP f(\xi) \xi\:d\xi
\\
&=&
\four_{|\xi|\rightarrow t} \left\{H_l(r\xi)\cdot \PP f(\xi)\xi \heav\right\}.
\end{eqnarray*}
Taking the $L^2_tL^2_{|x|\leq R}$ norm then gives, by Plancherel,
\begin{eqnarray}\label{stim1}
\nonumber
\|\ksol\|_{L^2_tL^2_{|x|\leq R}}^2
&=&\| \four_{|\xi|\rightarrow t} \left\{H_l(r\xi)\cdot \PP f(\xi)\xi \heav\right\}\|_{L^2_tL^2_{|x|\leq R}}^2
\\
\nonumber
&=&
\| H_l(r\xi)\cdot \PP f(\xi)\xi \heav\|_{L^2_\xi L^2_{|x|\leq R}}^2
\\
&=&
\int_0^{+\infty} \left(\int_0^{R}\left(H_l(r\xi) \cdot \PP f(\xi)\right)^2r\:dr\right)\xi^2\:d\xi.
\end{eqnarray}
We now consider the double integral above componentwise ($H_k$ is a matrix), and deal with each component separately. The first one (the second one is analogous) reads as
\begin{equation}\label{comp1}
\int_0^{+\infty} \left(\int_0^{R}\left( \chi_l (r\xi)\right)^2r\:dr\right)|\PP f(\xi)|^2\xi^2\:d\xi
\end{equation}
where we recall that $\chi_{l}$ denotes the radial component of the generalized eigenfunctions of the operator $\mathcal{D}_{A,l}$. To estimate inner integral in \eqref{comp1} we rely on the following estimate (see \cite{strichartz}) 
$$
\int_0^R J_k(r|\xi|)^2 rdr<\frac{CR}{|\xi|}
$$
which holds with a constant $C$ independent on $R$ and $k$. Thus, we eventually obtain 
%after a change of variable, this gives indeed
%$$
%\int_0^{R}\left( \chi_l (r\xi)\right)^2r\:dr= |\xi|^{-1}\int_0^{|\xi|R}\chi_l (\zeta)^2 \zeta d\zeta
%$$
%and thus we can eventually estimate \eqref{comp1} with
$$
\eqref{comp1}\lesssim C R\int_0^{+\infty} |\PP f(\xi)|^2 \xi\:d\xi =R\|f\|_{L^2}^2
$$
as $\PP $ is an isometry on $L^2$. Notice that the constant $C$ does not depend on $l$; therefore we can sum in decomposition \eqref{iso2} and by using the triangle inequality we obtain \eqref{locendec}.

\subsection{Proof of Therem \ref{ksstheo}}

The argument here turns out to be only a slight modification of the original one for the wave equation in \cite{kss} (see also \cite{jwy}) as, in fact, the only tools needed are the local smoothing estimate given by \eqref{locendec} and a standard energy estimate. We include the proof here anyway for the sake of completeness. 

We start from the case $\mu<- 1/2$: in this range we have, by applying estimate \eqref{locendec},
\begin{eqnarray*}
\| \langle x\rangle^\mu \sol \|_{L^2_tL^2_x}\lesssim \sum_{j\geq0}2^{j\mu}\|\sol\|_{L^2_tL^2_{|x|\leq 2^j}}
\lesssim
\sum_{j\geq 0}2^{j(\mu+1/2)}\|f\|_{L^2_x}
\leq
\|f\|_{L^2_x}.
\end{eqnarray*}
%(JUST SOME CARE: WE ARE DEALING WITH A SYSTEM...)
Now we deal with the case $-1/2\leq\mu\leq0$; we start by considering the case $T\leq 1$. Here, estimate \eqref{kss} is in fact weaker than the energy estimate
$$
\|\sol \|_{L^\infty_tL^2_x}\leq \|f\|_{L^2},
$$
so that we can immediately write
$$
\|\langle x\rangle^\mu \sol\|_{L^2_TL^2_x}\lesssim T^{1/2}\|\sol\|_{L^\infty_TL^2_x}\lesssim A_\mu(T)\|f\|_{L^2_x}
$$
as $\mu\geq-1/2$. In the region $T\geq2$ we can use energy estimate to have a control on the region $\{x:|x|\geq T\}$ as follows
$$
\|\langle x\rangle^\mu \sol\|_{L^2_T L^2_{|x|\geq T}} \lesssim T^\mu\|\sol \|_{L^2_TL^2_x}
\leq T^{1/2+\mu}\|f\|_{L^2_x}\lesssim A_\mu(T)\|f\|_{L^2_x}.
$$
For the remaining region, i.e. $T\in(1,2)$, we rely again on \eqref{locendec} to write
\begin{eqnarray*}
\|\langle x\rangle^\mu \sol\|_{L^2_TL^2_x}^2&\leq& \sum_{0\leq j\lesssim \ln(T)}2^{2j\mu}\|\sol\|^2_{L^2_TL^2_{|x|}\leq 2^j}
\\
&\leq&
\sum_{0\leq j\lesssim \ln(T)}2^{2j(\mu+1/2)}\|f\|^2_{L^2x}
\\
&\leq&
A_\mu(T)^2\|f\|_{L^2_x}^2.
\end{eqnarray*}

Eventually, we give a proof of \eqref{kss2}, which is a simple scaling argument from \eqref{kss}. If we considered for some $\lambda>0$ the rescaled function $f_\lambda(x)=\lambda^{\frac{n}2}f(\lambda x)$ so that $\left(e^{it\mathcal{D}_A}f_\lambda\right)x)=\lambda^\frac{n}2\left(e^{it\mathcal{D}_A\lambda t}f\right)(\lambda x)$, we have that estimate \eqref{kss} reads as
\begin{equation}\label{scaledkss}
\|\langle x\rangle^\mu\sol_\lambda\|_{L^2_{T/\lambda}L^2_x}\leq
\lambda^{-(\frac12+\mu)}T^{\frac12+\mu}\|f\|_{L^2}.
\end{equation}
Now we rewrite the (square of the) LHS of \eqref{scaledkss} by exploiting the change of variables $s=\lambda t$ and then $y=\lambda x$ to have
\begin{align*}
&\int_0^{T/\lambda} \| (1+|x|^2)^{\frac\mu2}\left(\sol_\lambda\right)(x)\|_{L^2_x}^2dt
\\
&\ \ \ 
=\lambda^{-\frac12}\int_0^{T} \| \lambda^{\frac32}(1+|x|^2)^{\frac\mu2}\left(e^{is\mathcal{D}_A}f\right)(\lambda x)\|_{L^2_x}^2ds
\\
&\ \ \ 
=\lambda^{-(\frac12+\mu)}\int_0^{T} \int_{\mathbb{R}^3} (\lambda^2+|y|^2)^\mu\left(e^{is\mathcal{D}_A}f\right)^2(y)dyds.
\end{align*}
Plugging identity above into \eqref{scaledkss} and sending $\lambda\rightarrow 0$ will give \eqref{kss2}, by dominated convergence.

%\subsection{Proof of Therem \ref{strichartz}}

\end{document}